\documentclass[11pt]{amsart}
\usepackage{geometry}                
\usepackage{graphicx}
\usepackage{amssymb}
\usepackage{epstopdf}
\usepackage{amsmath}
\usepackage{mathrsfs}
\usepackage[colorlinks=true, pdfstartview=FitV,linkcolor=blue,citecolor=blue,urlcolor=blue]{hyperref}
\newtheorem {theorem}    {Theorem}[section]

\newtheorem {lemma}      [theorem]    {Lemma}
\newtheorem {cor}  [theorem]    {Corollary}

\numberwithin{equation}{section}

\title{Residues and duality on semi-local two-dimensional adeles}

\author{Dongwen Liu}
\address{School of Mathematical Science, Zhejiang University, Hangzhou 310027, P.R. China}
\email{maliu@zju.edu.cn}


\thanks{This work was partially supported by NSFC 11201384.}

\begin{document}
\maketitle

\begin{abstract}
In this note, we establish a duality result under the residue paring between certain two-dimensional adelic spaces, which are associated to a closed point on an arithmetic surface.
\end{abstract}

\section{Introduction}

Higher dimensional adeles were introduced around 70s and 80s by Parshin \cite{P}
and Beilinson \cite{B}, who found among other things that for any quasi-coherent sheaf on a noetherian scheme, the cohomology of the associated adelic complex coincides with its sheaf cohomology. Their interesting idea and elegant construction of the adelic resolution provide an explicit adelic approach to study the geometry and arithmetic of higher dimensional schemes, algebraic or arithmetic. A  large variety of developments originated from this idea since then, although we will only mention a couple of typical examples below for illustration.  We apologize for not mentioning many other important related works on this subject, which we are not able to recall in this short note.

For instance, Yekutieli in the monograph \cite{Y} gives an adelic construction for the Grothendieck residue complex for algebraic varieties over a perfect field; Osipov \cite{O}
applies adelic constructions of relative residues and symbols to the Gysin map along a projective morphism from a smooth algebraic surface to a smooth algebraic curve over a perfect field; various Weil and Parshin reciprocity laws are established in \cite{OZ} via categorical method; the arithmetic case is considered by Morrow in \cite{M1, M2}, where the dualizing sheaves 
and Grothendieck's trace map for arithmetic surfaces are described explicitly using adeles; following Morrow's treatment we consider the multiplicative analog for arithmetic surfaces in \cite{L} and establish several reciprocities. This subject definitely deserves further investigations, and we would like to take this opportunity to mention that the generalization of adelic methods to higher dimensional arithmetic schemes will be carried out in a forthcoming project, including various important topics such as Grothendieck's residue complex, relative trace maps via local cohomology, multiplicative symbols, reciprocity laws, Gysin morphisms and so on.

Moreover, we currently have a renewed interest due to the work of Osipov and Parshin in \cite{OP1, OP2}, where a theory of harmonic analysis in dimension two is developed and leads to an analytic proof of the Riemann-Roch theorem for smooth projective algebraic surfaces over a finite field.

This paper grows out of a recent joint work \cite{LZ} with Y. Zhu, where we reformulate aforementioned Osipov-Parshin's theory in a slightly more canonical way, extend the notion of Weil index and as applications establish certain quadratic reciprocity laws on  arithmetic surfaces. For that purpose, a necessary piece of ingredient is the residue pairing between certain two-dimensional adelic spaces as well as some duality statements. Our main results, Theorem \ref{thm} and Corollary \ref{cor}, will be used in the analytic proof of quadratic reciprocities in \cite{LZ}.
 
 Let us briefly outline the main result as well as the structure of the paper. Let $X$ be an arithmetic surface (see Section 2 for details) over a Dedekind domain $\mathcal O_K$ of characteristic zero with finite residue fields. We consider a semi-local situation, namely we fix a closed point $x\in X$ lying over a finite place $s$ of $\mathcal O_K$, and take various formal curves $y$ containing $x$. Let $K(X)_x=\textrm{Frac~}\widehat{\mathcal O}_{X, x}$ and $K_s=\textrm {Frac }\widehat{\mathcal O}_{K,s}$. We define the adele ring ${\bf A}_{X,x}$ at $x$ and the adelic space of continuous relative differential forms $\Omega^{\rm cts}_{{\bf A}_{X,x}/K_s}$, between which there is a canonical residue pairing 
 \[
 {\bf A}_{X,x}\times \Omega^{\rm cts}_{{\bf A}_{X,x}/K_s}\to K_s, \quad (f, \omega)\mapsto \textrm{Res}_x(f\omega):=\sum_{y\subset^f X, y\ni x}\textrm{Res}_{x,y}(f_y\omega_y)
 \]
 where the last sum is taken over all formal curves $y\subset^f X$ containing $x$. Here and thereafter we write $y\subset^f  X$ to indicate that $y$ is only a formal curve, which may not be global. In \cite{M1, M2} it is proved that the residue pairing is trivial restricted on the rational points 
 $K(X)_x\times \Omega^{\rm cts}_{K(X)_x/K_s}$. Our main result states that,  
 $K(X)_x$ and $ \Omega^{\rm cts}_{K(X)_x/K_s}$ are in fact mutually the full annihilators under the residue paring at $x$. This suggests, for example, that the arithmetic quotient 
 ${\bf A}_{X,x}/K(X)_x$ can be thought of as the ``Pontryagin dual" of the space of rational differential forms $\Omega^{\rm cts}_{K(X)_x/K_s}$ in some suitable sense.

In Section 2 we recall the construction of residue maps in dimension two as well as the reciprocity from \cite{M1, M2}. In Section 3 we formulate the main result in its algebraic and geometric versions, and give a purely algebraic and elementary proof, which only makes use of the explicit structures of two-dimensional local fields as well as several standard facts from algebraic number theory. Our treatment is thus self-contained modulo these well-known results on local fields.

{\bf Acknowledgement.} The author is in debt to Y. Zhu for several discussions during their joint work \cite{LZ}.  He is grateful to the hospitality of HKUST, in a visit to which the work was started.

\section{Residues in dimension two}

In this section,  as preliminaries we briefly recall the construction of residue maps in dimension two from \cite{M1, M2}. A two-dimensional local field under our concern is a complete discrete valuation field $F$, whose residue field $\overline{F}$ is a complete field of discrete valuation with finite residue field. Let $F$ be such a field with characteristic zero, and $K\subset F$ be a local field contained in $F$ such that 
Frac$(K\cap {\mathcal O}_F)=K$. Let $k_{F,K}$ be the algebraic closure of $K$ in $F$, which is finite over $K$ hence a local field as well. We fix the subfield $K$ and write $k_F$ instead of $k_{F,K}$ for simplicity.

For a module $M$ over a local ring $A$ with maximal ideal $\frak{m}_A$, we write $M^{\rm sep}=M/\bigcap_{n\geq 0}\frak{m}_A^nM$ for the maximal separable quotient of $M$. 
Define the space of continuous relative differential forms $\Omega^{\rm cts}_{F/K}=\Omega^{\rm sep}_{\mathcal O_F/ K\cap \mathcal O_F}\otimes_{\mathcal O_F} F$.
If $F$ is of equal characteristic, then a choice of a local parameter $t$ leads to a unique $k_F$-isomorphism $F\cong k_F((t))$, and one has $K\subset \mathcal O_F$, $\Omega^{\rm sep}_{\mathcal O_F/K}=\mathcal O_F dt$. The relative residue map is defined by
\[
\textrm{Res}_{F/K}: \Omega^{\rm cts}_{F/K}\to K, \quad \sum_n a_n t^n dt\mapsto \textrm{Tr}_{k_F/K}(a_{-1}),
\]
which does not depend on the choice of $t$.

If $F$ is of mixed characteristic, then by Cohen's structure theory \cite{C} there is a two dimensional local field $L$ inside $F$ such that $F/L$ is finite, $\overline{F}=\overline{L}$, $k_F=k_L$ and $L$ is $k_L$-isomorphic to 
\[
k_L\{\{t\}\}:=\left\{\left.\sum_{n\in\mathbf{Z}} a_n t^n\right| a_n\in k_L \textrm{ is bounded and }\lim_{n\to-\infty}a_n=0\right\}.
\]
Such $L$ is called \textit{standard}.  For standard fields one has $K\cap \mathcal O_L=\mathcal O_K$ and $\Omega^{\rm sep}_{\mathcal O_L/\mathcal O_K}=\mathcal O_L dt$. The relative residue map is defined by
\[
\textrm{Res}_{L/K}: \Omega^{\rm cts}_{L/K}\to K, \quad \sum_n a_n t^n dt\mapsto -\textrm{Tr}_{k_L/K}(a_{-1}).
\]
Note the minus sign in the above definition. In general we define 
\[
\textrm{Res}_{F/K}=\textrm{Res}_{L/K}\circ \textrm{Tr}_{F/L}: \Omega^{\rm cts}_{F/K}\to K,
\]
which is independent of the choice of $L$ and the $k_L$-isomorphism $L\cong k_L\{\{t\}\}$.

It is proved in \cite{M1} that the residue maps defined above are functorial, i.e. they commute with trace maps, and in \cite{M2} they are shown to be continuous with respect to the discrete valuation topology. The reciprocity formulated below as well as its geometrization are established in \cite{M1, M2}. 

Let $A$ be a two-dimensional normal complete local ring of characteristic 
zero with finite residue field of characteristic $p$, and $E$ be the fractional field of $A$.
Let $K$ be a $p$-adic local field of characteristic zero,  such that its ring of integers ${\mathcal O}_K$ is contained in $A$.   We write  $\frak{P}\lhd^1 A$ for a height one prime ideal $\frak{P}$ of $A$, and let $E_{\frak{P}}$ be the fractional field of $\widehat{A_\frak{P}}$, the completion of $A_\frak{P}$ with respect to the discrete valuation induced by $\frak{P}A_\frak{P}$. Let us denote ${\Omega}^{\textrm{sep}}_{A/{\mathcal O}_K}\otimes_A E$ by $\Omega^{\textrm{cts}}_{E/K}$, which is a quotient of $\Omega_{E/K}$. For simplicity we write 
\[
\textrm{Res}_\frak{P}=\textrm{Res}_{E_{\frak{P}}/K}:  \Omega^{\rm cts}_{E_{\frak{P}}/K}\to K
\]
and by abuse of notation we also denote by $\textrm{Res}_\frak{P}$ its restriction to $\Omega^{\rm cts}_{E/K}$ under the natural embedding $\Omega^{\rm cts}_{E/K}\hookrightarrow \Omega^{\rm cts}_{E_\frak{P}/K}$, called the \textit{residue map at} $\frak{P}$. Then the reciprocity proved in \cite{M1} states that

\begin{theorem}\label{reci}
Let $\omega\in\Omega^{\rm{cts}}_{E/K}$. Then the residue ${\rm Res}_\frak{P}(\omega)=0$ for almost all height one primes $\frak{P}$ of $A$ and in $K$ one has
\[
{\rm Res}_K(\omega):=\sum_{\frak{P}\lhd^1 A}{\rm Res}_\frak{P}(\omega)=0.
\]
\end{theorem}

As outlined in the Introduction, one has geometric interpretation for the above reciprocity in terms of arithmetic surfaces.  Let us briefly recall the formulation and refer the reader to \cite{M2} for more details. See also \cite{L} for similar discussions in the context of multiplicative symbols.

Let $\mathcal O_K$ be a Dedekind domain of characteristic zero with finite residue fields. Let $X$ be a two-dimensional, normal scheme, flat and projective over $S=\textrm{Spec }\mathcal O_K$, whose generic fibre is one-dimensional and irreducible. Let $x$ be a codimension two closed point on $X$ lying over a closed point $s\in S$. Then the complete local ring $A:=\widehat{\mathcal O}_{X,x}$ satisfies the properties assumed in the above, and contains the discrete valuation ring $\widehat{\mathcal O}_{K,s}$. To each formal curve $y$ containing $x$, i.e. a height one prime ideal $y$ of $A$, one may associate a two-dimensional local field $K_{x,y}$ and define the local residue map
\[
\textrm{Res}_{x,y}=\textrm{Res}_{K_{x,y}/K_s}: \Omega^{\rm cts}_{K_{x,y}/K_s}\to K_s.
\]
Let $K(X)_x=\textrm{Frac }A\hookrightarrow K_{x,y}$ and $K_s=\textrm{Frac }\widehat{\mathcal O}_{K,s}$. Then the reciprocity above can be reformulated in a  geometric translation as 

\begin{theorem} \label{geo}
Let $\omega\in \Omega^{\rm{cts}}_{K(X)_x/K_s}$, and $x\in X$ be a closed point lying over $s\in S$. Then ${\rm Res}_{x,y}(\omega)=0$ for almost all formal curves $y\subset ^f X$ containing $x$ and
\[
{\rm Res}_x(\omega):=\sum_{y\subset^f X, y\ni x}{\rm Res}_{x,y}(\omega)=0.
\]
\end{theorem}

The above theorem also has a variant form where the sum is over global curves. Let $Y\subset X$ be an irreducible curve passing through $x$, which in general may have several formal branches $y_1, \ldots, y_n$ in Spec $A$. The curve $Y$ defines a prime ideal of $\mathcal{O}_{X,x}$, which we again denote by $Y$ without causing any confusion. Then we may form a finite product of two-dimensional local fields 
\[
K_{x,Y}:=\prod_{y\lhd^1 A, y|Y } K_{x,y}.
\]
Define the residue map
\[
\textrm{Res}_{x,Y}=\sum_{y\lhd^1 A, y|Y}\textrm{Res}_{K_{x,y}/K_s}: \Omega^{\rm cts}_{K_{x,Y}/K_s}\to K_s.
\]
Then one can easily formulate an analog of Theorem \ref{geo} by taking $\omega\in
\Omega_{K(X)/K}$ and the sum of residues Res$_{x,Y}(\omega)$ over all global curves $Y\subset X$ containing $x$.

\section{Main result}
We adopt the notations and formulations in the last section. Again let $A$ be a two-dimensional local ring with properties specified as before and let $E$ be its fractional field. Define the adele ring ${\bf A}_E$ of $E$ to be the restricted product $\prod'_{\frak{P}\lhd^1 A}E_\frak{P}$ with respect to $\widehat{A_\frak{P}}$'s, and let $\Omega^{\rm cts}_{{\bf A}_E/K}$ be the restricted product
$\prod'_{\frak{P}\lhd^1 A} \Omega^{\textrm{cts}}_{E_\frak{P}/K}$ with respect to $\Omega^{\textrm{sep}}_{{\mathcal O}_{E_\frak{P}}/K\cap {\mathcal O}_{E_\frak{P}}}$'s.  Then we have the pairing
\[
{\bf A}_E\times \Omega^{\rm cts}_{{\bf A}_E/K}\to K,\quad (f, \omega)\mapsto \textrm{Res}_{E/K}(f\omega):=\sum_{\frak{P}\lhd^1 A} \textrm{Res}_\frak{P} (f_\frak{P}\omega_\frak{P}),
\]
which is actually a finite sum.  Notice that we have diagonal embeddings $E\hookrightarrow {\bf A}_E$ and $\Omega^{\textrm{cts}}_{E/K}\hookrightarrow \Omega^{\rm cts}_{{\bf A}_E/K}$. Then by Theorem \ref{reci}, the residue paring is trivial restricted to the space of rational points $E\times \Omega^{\rm cts}_{E/K}$.

Our purpose is to prove the following result, which identifies the arithmetic quotients ${\bf A}_E/E$ and $\Omega^{\rm cts}_{{\bf A}_E/K}/ \Omega^{\rm cts}_{E/K}$ as ``Pontryagin duals" of $\Omega^{\rm cts}_{E/K}$ and $E$ respectively in some appropriate sense.

\begin{theorem}\label{thm}
$E$ and $\Omega^{\mathrm{cts}}_{E/K}$ are mutual annihilators under the residue pairing.
\end{theorem}

Let us rephrase the theorem in a geometric way before giving its proof. We again follow the setup and notations in the last section. Consider a formal curve $y\subset^f X$ containing $x$. Write $\mathcal O_{x,y}=\mathcal O_{K_{x,y}}$. Define the adelic spaces
\[
{\bf A}_{X,x}={\prod}'_{y\subset^f X, y\ni x}K_{x,y}, 
\]
the restricted product with respect to $\mathcal O_{x,y}$'s, and
\[
\Omega^{\rm cts}_{{\bf A}_{X,x}/K_s}={\prod}'_{y\subset^f X, y\ni x} \Omega^{\rm cts}_{K_{x,y}/K_s},
\]
the restricted product with respect to $\Omega^{\rm sep}_{\mathcal O_{x,y}/K_s\cap \mathcal O_{x,y}}$'s. Gluing all these pieces of constructions we obtain the residue paring at $x$,
\[
{\bf A}_{X,x}\times \Omega^{\rm cts}_{{\bf A}_{X,x}/K_s}\to K_s, \quad (f, \omega)\mapsto \textrm{Res}_x(f\omega):=\sum_{y\subset^f X, y\ni x}\textrm{Res}_{x,y}(f_y\omega_y),
\]
where $\textrm{Res}_{x,y}$ is as defined in the last section. Consider the diagonal embeddings $K(X)_x\hookrightarrow {\bf A}_{X,x}$ and $\Omega^{\rm cts}_{K(X)_x/K_s}\hookrightarrow\Omega^{\rm cts}_{{\bf A}_{X,x}/K_s}$. Theorem \ref{geo} states that the residue pairing is trivial restricted on $K(X)_x\times \Omega^{\rm cts}_{K(X)_x/K_s}$. Now we have the following geometric reformulation of Theorem \ref{thm}.

\begin{cor}\label{cor}
$K(X)_x$ and $\Omega^{\rm cts}_{K(X)_x/K_s}$ are mutual annihilators under the residue  pairing at $x$.
\end{cor}

We shall reduce the proof of Theorem \ref{thm} to the case of a two-dimensional regular local ring. In fact by Cohen's structure theorem \cite{C}, $A$ has a subring $B$ which contains ${\mathcal O}_K$ and is ${\mathcal O}_K$-isomorphic to ${\mathcal O}_K[[t]]$. Then $B$ is a two-dimensional complete regular local ring with maximal ideal $\mathfrak{m}_B=\langle \pi_K, t\rangle$, where $\pi_K$ is a local parameter of $K$. Recall that by Weierstrass's preparation theorem (cf. \cite{W}), a height one prime of $B$ is generated by either $\pi_K$ or an irreducible, distinguished polynomial (i.e. of the form $t^l+a_1 t^{l-1}+\cdots+ a_l$ with $a_i\in \frak{m}_K$). Let $F$ be the fractional field of $B$. We first prove that

\begin{lemma}
Theorem \ref{thm} holds for $B$, i.e. $F$ and $\Omega^{\mathrm{cts}}_{F/K}$ are mutual annihilators under the residue pairing
\[
{\bf A}_F\times \Omega^{\rm cts}_{{\bf A}_F/K}\to K,\quad (f, \omega)\mapsto \mathrm{Res}_{F/K}(f\omega):=\sum_{\frak{p}\lhd^1 B} \mathrm{Res}_\frak{p} (f_\frak{p}\omega_\frak{p}).
\]
\end{lemma}

\begin{proof}
First assume that for a fixed $\omega=(\omega_\frak{p})_\frak{p}\in\Omega^{\rm cts}_{{\bf A}_F/K}$ we have $\textrm{Res}_{F/K}(f\omega)=0$ for any $f\in F$. We need to show that $\omega\in \Omega^{\textrm{cts}}_{F/K}$.  It is equivalent to replace $\omega$ by $\omega'=f\omega$ with any $f\in F^\times$ and prove the conclusion for $\omega'$. Hence w.l.o.g. we may assume that $\omega_\frak{p}\in \Omega^{\textrm{sep}}_{{\mathcal O}_{F_\frak{p}}/K}$ for any $\frak{p}\neq \pi_KB$. For any such $\frak{p}$, which corresponds to an irreducible distinguished polynomial $P=P(t)\in{\mathcal O}_K[t]$, we may
fix an isomorphism $F_\frak{p}\cong k_\frak{p}((t_\frak{p}))$ where $k_\frak{p}\cong K[t]/(P(t))$ and $t_\frak{p}$ is a local parameter of $\frak{p}$. Here we may and do choose $t_\frak{p}=P(t)$. Then
\[
 \Omega^{\textrm{cts}}_{F_\frak{p}/K}=F_\frak{p}dt_\frak{p}=k_\frak{p}((t_\frak{p}))dt_\frak{p}.
\]
Denote $\pi_K B$ by $\frak{p}_0$. Then it is clear that $P(t)\in {\mathcal O}^\times_{F_{\frak{p}'}}$ for any $\frak{p}'\neq \frak{p}, \frak{p}_0$. 
The image of $t$ under the embedding $F\hookrightarrow k_\frak{p}((t_\frak{p}))$ is a root of the equation $P(t)=t_\frak{p}$, which a priori exists by Hensel's lemma. One can easily show that it is of the form
\[
t= \sum^\infty_{i=0} c_i t_\frak{p}^i,\quad c_i\in k_\frak{p},
\]
where $c_0\in k_\frak{p}$ is a root of $P(t)$. Note that $t\in {\mathcal O}_{F_{\frak{p}'}}$ for any $\frak{p}'\neq \frak{p}_0$.

If we write
\[
\omega_\frak{p}=\sum_{i\geq 0} a_{i,\frak{p}}t_\frak{p}^i dt_\frak{p},\quad a_{i,\frak{p}}\in k_\frak{p},
\]
then for any $n\geq 0$,
\[
\textrm{Res}_{F/K}(P^{-1} t^n\omega)=\textrm{Res}_\frak{p}(P^{-1}t^n\omega_\frak{p})+\textrm{Res}_{\frak{p}_0}(P^{-1}t^n\omega_{\frak{p}_0})=0,
\]
which implies that
\[
\textrm{Res}_\frak{p}(P^{-1}t^n\omega_\frak{p})=-\textrm{Res}_{\frak{p}_0}(P^{-1}t^n\omega_{\frak{p}_0}).
\]
By definition it is straightforward to check that
\[
\textrm{Res}_\frak{p}(P^{-1}t^n\omega_\frak{p})=\mathrm{Tr}_{k_\frak{p}/K}(c_0^n a_{0,\frak{p}}).
\]
Since $\{1, c_0, \cdots, c_0^{\deg P-1}\}$ form a basis of $k_\frak{p}$ over $K$, the last two equations imply that $\mathrm{Tr}_{k_\frak{p}/K}(\lambda a_{0,\frak{p}})$ is determined by
$\omega_{\frak{p}_0}$ for any $\lambda\in k_\frak{p}$. Using non-degeneracy of the pairing 
\[
k_\frak{p}\times k_\frak{p}\to K,\quad (a,b)\mapsto \mathrm{Tr}_{k_\frak{p}/K}(ab),
\]
 we see that $a_{0,\frak{p}}$ is uniquely determined by $\omega_{\frak{p}_0}$. By induction, similar arguments apply for $\textrm{Res}_{F/K}(P^{-i-1}t^n\omega)$ and imply that
 $a_{i,\frak{p}}$ is determined by $\omega_{\frak{p}_0}$ for any $i\geq 0$.

Therefore we have deduced that
$\omega_\frak{p}$ is determined by $\omega_{\frak{p}_0}$ for arbitrary $\frak{p}\neq \frak{p}_0$. In other words, if $\omega, \omega'$ both annihilate $F$ and
$\omega_{\frak{p}_0}=\omega_{\frak{p}_0}'$, then $\omega=\omega'$. Hence we are reduced to showing that $\omega_{\frak{p}_0}\in\Omega^{\textrm{cts}}_{F/K}$. Recall that
$\Omega^{\rm cts}_{F_{\frak{p}_0}/K}=K\{\{t\}\}dt$. If we write
\[
\omega_{\frak{p}_0}=\sum_{i\in\bf{Z}}a_{i,\frak{p}_0}t^i dt,\quad a_{i,\frak{p}_0}\in K,
\]
then similarly as above we have
\[
a_{-i,\frak{p}_0}=-\textrm{Res}_{\frak{p}_0}(t^{i-1}\omega_{\frak{p}_0})=\textrm{Res}_{(t)}(t^{i-1}\omega_{(t)})=0,\quad \forall i\geq 1.
\]
Hence we obtain
\[
\omega_{\frak{p}_0}=\sum_{i\geq 0}a_{i,\frak{p}_0}t^i dt\in \Omega^{\textrm{cts}}_{F/K}.
\]
This proves that $\omega\in \Omega^{\textrm{cts}}_{F/K}$.

In the same way one can show that $F$ is the annihilator of $\Omega^{\textrm{cts}}_{F/K}$.
\end{proof}

Let us now prove the general case.  At this time we shall only prove that $E$ is the annihilator of $\Omega^\mathrm{sep}_{E/K}$, and leave the proof of the dual statement as an exercise which is slightly more technical but can be essentially handled in a similar manner.

Assume that for a fixed $f=(f_\frak{P})_\frak{P}\in {\bf A}_E$ we have $\mathrm{Res}(f\omega)=0$ for any $\omega\in \Omega^\mathrm{cts}_{E/K}$. We have to show that $f\in E$. Consider the  natural trace maps ${\bf A}_E\to {\bf A}_F$ and $\Omega^{\rm cts}_{{\bf A}_E/K}\to \Omega^{\rm cts}_{{\bf A}_F/K}$, both of which will be denoted by  $\mathrm{Tr}_{E/F}$. By functoriality of the residues with respect to the trace maps, 
\[\mathrm{Res}_{F/K}(\mathrm{Tr}_{E/F}(f)\omega')=\mathrm{Res}_{F/K}(\mathrm{Tr}_{E/F}(f\omega'))=\mathrm{Res}_{E/K}(f\omega')=0
\]
for any $\omega'\in \Omega^\mathrm{cts}_{F/K}$. From the last lemma it follows that
$\mathrm{Tr}_{E/F}(f)\in F$. Replacing $f$ by $fg$ with $g\in E$, similarly we have 
$\mathrm{Tr}_{E/F}(fg)\in F$. Hence we are reduced to proving the following result, which is a consequence of several standard facts from algebraic number theory. It should be a fairly standard result but it does not seem to be widely available in literature, so we give a detailed proof for completeness.

\begin{lemma}
Let $A/B$ be a finite extension of two-dimensional normal complete local rings of characteristic zero and $E/F$ be the extension of their fractional fields. Assume that $f\in {\bf A}_E$ satisfies $\mathrm{Tr}_{E/F}(fg)\in F$ for any $g\in E$. Then $f\in E$.
\end{lemma}

\begin{proof}
Take a finite extension $A'/A$ of two-dimensional normal complete local rings, and let $E'$ be the fractional field of $A'$. Assume that the lemma holds for $A'/B$.
Take $f\in {\bf A}_E$ such that $\mathrm{Tr}_{E/F}(fg)\in F$ for any $g\in E$. Then by functoriality of the trace map we have
\[
\textrm{Tr}_{E'/F}(fg')=\textrm{Tr}_{E/F}\textrm{Tr}_{E'/E}(fg')=\textrm{Tr}_{E/F}(f\textrm{Tr}_{E'/E}(g'))\in F
\]
for any $g'\in E'$. By our assumption this implies that $f\in E'$, hence $f\in{\bf A}_E\cap E'=E$. Hence the lemma holds for $A/B$ as well. Therefore by passing to a finite extension we may and do assume that $E$ is Galois over $F$.

We proceed by several steps, starting from the ``rationality" of $f_\frak{P}$ for each $\frak{P}\lhd^1 A$. That is, we first show that $f_\frak{P}\in E$. Let $\{\beta_1,\ldots, \beta_n\}$ be a basis of $E$ over $F$, where $n=[E:F]$.
Fix $\frak{p}\lhd^1 B$ and let $\{\sigma_1,\ldots, \sigma_n\}$ be the multiset union of $G_\frak{P}:=\mathrm{Gal}(E_\frak{P}/F_\frak{p})$ for all $\frak{P}|\frak{p}$. Recall that the $\frak{p}$-component of $\mathrm{Tr}_{E/F}(f\beta_j)$ is 
\[
\sum_{\frak{P}|\frak{p}}\mathrm{Tr}_{E_\frak{P}/F_\frak{p}}(f_\frak{P}\beta_j)=\sum_{\frak{P}|\frak{p}}\sum_{\sigma\in G_\frak{P}}\sigma (f_\frak{P})\sigma(\beta_j)=\sum^n_{i=1}\sigma_i(f_{\frak{P}_i})\sigma_i(\beta_j).
\]
Here $\{\frak{P}_1,\ldots, \frak{P}_n\}$ is the multiset where all the $\frak{P}$'s, $\frak{P}|\frak{p}$, appear with equal multiplicities such that $\sigma_i \in G_{\frak{P}_i}$. The matrix $D=(\sigma_i(\beta_j))_{1\leq i,j\leq n}$ is nonsingular and rational over $E$. Indeed, from $\mathrm{Tr}_{E/F}=\sum_{\frak{P}|\frak{p}}\mathrm{Tr}_{E_\frak{P}/F_\frak{p}}$ one verifies that
$D^T D=(\mathrm{Tr}_{E/F}(\beta_i\beta_j))$,
whose determinant is the discriminant $D(\beta_1,\ldots, \beta_n)\neq 0$. Now by assumption $\mathrm{Tr}_{E/F}(f\beta_j)\in F$ for all $j$, so $\sigma_i(f_{\frak{P}_i})\in E$ for all $i$. In particular $f_\frak{P}\in E$ for all $\frak{P}|\frak{p}$. 

It remains to prove that $f_\frak{P}$'s are equal for all $\frak{P}\lhd^1 A$. It turns out that we only need to prove a weaker result that $f_\frak{P}$'s are equal for all $\frak{P}|\frak{p}$ with $\frak{p}$ fixed. In fact, assume this is true and let $f_\frak{p}\in E$ be the common value of $f_\frak{P}$'s with $\frak{P}|\frak{p}$. Then by assumption $\mathrm{Tr}_{E/F}(fg)=(\mathrm{Tr}_{E/F}(f_\frak{p}g))_\frak{p}\in F$ for any $g\in E$, which implies that $\mathrm{Tr}_{E/F}(f_\frak{p}g)$'s are equal for all $\frak{p}$. The $f_\frak{p}$'s must be also equal for all $\frak{p}$ due to the non-degeneracy of the paring 
\[
E\times E\to F,\quad (a,b)\mapsto \mathrm{Tr}_{E/F}(ab).
\]
This proves that $f$ lies in $E$.

Finally let us show that $f_\frak{P}$'s are equal for all $\frak{P}|\frak{p}$ with $\frak{p}$ fixed.  Let $l$ be the number of these $\frak{P}$'s. Consider the $F$-linear map
\[
T: E^l=\bigoplus_{\frak{P}|\frak{p}}E\to E^n, \quad (x_\frak{P})_\frak{P}\mapsto \big(\sum_{\frak{P}|\frak{p}}\mathrm{Tr}_{E_\frak{P}/F_\frak{p}}(x_\frak{P}\beta_i)\big)_{i=1,\ldots, n}.
\]
Let $E_d$  be the copy of $E$ diagonally embedded into $E^l$. Then apparently $E_d\subset T^{-1}(F^n)$, and what we need to prove is equivalent to that $E_d=T^{-1}(F^n)$. Comparing the $F$-dimensions, it suffices to show that $T$ 
is injective.
Assume that $T((x_\frak{P})_\frak{P})= 0$ but $x_{\frak{P}'}\neq 0$ for some $\frak{P}'|\frak{p}$. By the approximation theorem for discrete
valuations, there exists $g\in E$ such that \[
x_{\frak{P}'}g\equiv \frac{1}{[E_\frak{P}: F_\frak{p}] }\mod \frak{P}'
\] and $x_\frak{P}g\in \frak{P}$ for $\frak{P}\neq \frak{P}'$. It follows that
\[
\sum_{\frak{P}|\frak{p}}\mathrm{Tr}_{E_\frak{P}/F_\frak{p}}(x_\frak{P}g)\equiv 1 \mod \frak{p}\widehat{B_\frak{p}}.
\]
In particular, above is nonzero, which leads to a contradiction. This proves that $T$ is injective hence finishes the proof of the lemma.
\end{proof}

\end{document}